\documentclass[11pt,reqno]{article}
\usepackage{amssymb}
\usepackage{amsmath}
\usepackage{amsthm}
\usepackage{amsfonts}
\usepackage{url}
\usepackage{breakurl}
\usepackage{hyperref}
\usepackage{comment}

\renewcommand{\ge}{\geqslant}
\renewcommand{\le}{\leqslant}

\newcommand{\R}{{\mathbb R}}

\newcommand{\dif}{{\,d}}
\newcommand{\NZ}{{\mathcal F}}

\theoremstyle{plain}
\newtheorem{theorem}{Theorem}
\newtheorem{corollary}{Corollary}
\newtheorem{lemma}{Lemma}

\theoremstyle{definition}

\begin{document}
\bibliographystyle{plain}
\title{A harmonic sum over nontrivial zeros
of the Riemann zeta-function
}
\author
{Richard P.\ Brent\footnote{Australian National University,
Canberra, Australia
{\tt <Hlimit@rpbrent.com>}},\; 
David J.\ Platt\footnote{School of Mathematics, University of Bristol,
Bristol, UK 
{\tt <dave.platt@bris.ac.uk>}}\;
and Timothy S.\ Trudgian\footnote{School of Science, Univ.\ of NSW,
Canberra, Australia 
{\tt <t.trudgian@adfa.edu.au>}}
}
\maketitle

\vspace*{-10pt}
\begin{abstract}
\vspace*{5pt}
\noindent
We consider the sum $\sum 1/\gamma$, where $\gamma$ ranges over
the ordinates of nontrivial zeros of the Riemann zeta-function
in an interval $(0,T]$, and consider the behaviour of the
sum as $T \to\infty$. We show that, after subtracting a smooth
approximation 
$\frac{1}{4\pi} \log^2(T/2\pi),$
the sum tends to a limit $H \approx -0.0171594$ 
which can be expressed as an integral. 
We calculate $H$ to high accuracy, using a method which has
error $O((\log T)/T^2)$.
Our results improve on earlier results by Hassani and other authors.
\end{abstract}

\section{Introduction}				\label{sec:Intro}

Let the nontrivial zeros of the Riemann zeta-function $\zeta(s)$ be
denoted by $\rho = \sigma + i\gamma$. In order of increasing height,
the ordinates of these zeros in the upper half-plane are
$\gamma_1 \approx 14.13  < \gamma_2 < \gamma_3 < \cdots$.
Define
\[
G(T) := \sum_{0<\gamma\le T}1/\gamma,
\]
where multiple zeros (if they exist) are weighted according to their
multiplicity.  We consider the behaviour of $G(T)$ as
$T \to \infty$. 
Answering a question of Hassani~\cite{Hassani-2016}, we show
in Theorem~\ref{thm:H} of \S\ref{sec:limit} that there exists
\begin{equation}			\label{eq:H}
H := 
 \lim_{T\to\infty}\left(G(T) - \frac{\log^2(T/2\pi)}{4\pi}\right).
\end{equation}

There is an analogy with the harmonic series $\sum 1/n$,
which 
appears in the usual definition of Euler's constant:
\[
C :=  \lim_{N \to \infty}\left(\sum_{n=1}^N \frac{1}{n} - \log N\right)
  = 0.577215\cdots.
\]
It is well-known that one can compute $C$ accurately using Euler--Maclaurin
summation or faster algorithms, see \cite{rpb256,rpb049,Demailly2017} 
and the references given there.
However, it is not so easy to compute $H$ accurately, because of the
irregular spacing of the nontrivial zeros of $\zeta(s)$,
for which see~\cite{Odlyzko1987}.

In \S\ref{sec:extrapolation} we consider numerical approximation of~$H$,
after giving some preliminary lemmas in~\S\ref{sec:lemmas}.
If the 
definition~\eqref{eq:H} is used directly with the zeros up to height~$T$,
then the error is $O((\log T)/T)$.
In Theorem~\ref{thm:accelerated} 
we show how to improve this, without much extra computation,
to $O((\log T)/T^2)$.
In Corollary \ref{cor:huck} we give an explicit bound on $H$
with error of order $10^{-18}$.

Finally, in \S\ref{sec:related}, we comment briefly on related
results in the literature.
\section{Existence of the limit}		\label{sec:limit}

Before proving Theorem~\ref{thm:H}, we define some notation.
Let $\NZ$ denote the set of positive ordinates of zeros of $\zeta(s)$.
Following Titchmarsh~\cite[\S9.2--\S9.3]{Titch},
if $0 < T\not\in\NZ$,
then we let $N(T)$ denote the
number of zeros $\beta+i\gamma$ of $\zeta(s)$ with
$0 < \gamma \le T$, and
$S(T)$ denote the value of \hbox{$\pi^{-1}\arg\zeta(\frac12+iT)$} obtained by
continuous variation along the straight lines joining
$2$, \hbox{$2+iT$}, and $\frac12+iT$, starting with the value $0$.
If $T\in\NZ$, we could take
$S(T) = \lim_{\delta\to 0}[S(T-\delta)+S(T+\delta)]/2$,
and similarly for $N(T)$, but we avoid this exceptional case.
Note that $N(T)$ and $S(T)$ are piecewise continuous, with jumps at
$T\in\NZ$.

By~\cite[Thm.~9.3]{Titch}, we have
$N(T) = L(T) + Q(T)$, where
\[
L(T) = \frac{T}{2\pi}\left(\log\left(\frac{T}{2\pi}\right)-1\right) +\frac78\,
\text{, and } Q(T) = S(T) + O(1/T).
\]
An explicit bound from Trudgian~\cite[Cor.~1]{Trudgian-2014} is
\begin{equation}			\label{eq:Q}
Q(T) = S(T) + \frac{0.2\vartheta}{T}\,
\end{equation}
where (here and elsewhere) $\vartheta\in\R$ satisfies $|\vartheta|\le 1$.

Let $S_1(T) := \int_0^T S(t)\,dt$. 
By \cite[Thm.~9.4 and Thm.~9.9(A)]{Titch}, we have
$S(T) = O(\log T)$ and
$S_1(T) = O(\log T)$, and it follows from~\eqref{eq:Q} that
\hbox{$Q(T) = O(\log T)$} also.

Explicit bounds on $S_1(T)$ are known.
For certain constants $c$, $A_0\ge 0$, $A_1\ge 0$, and $T_0 > 0$, 
there is a bound
\begin{equation}			\label{eq:S1bound}
|S_1(T)-c| \le A_0 + A_1\log T \text{ for all } T \ge T_0.
\end{equation}
{From}~\cite[Thm.~2.2]{Trudgian-2011}, we could
take $c=S_1(168\pi)$, $A_0=2.067$, $A_1=0.059$, and $T_0 = 168\pi$.
However, a small computation shows that~\eqref{eq:S1bound} also
holds for $T\in[2\pi,168\pi]$. Hence, we take $T_0 = 2\pi$
in~\eqref{eq:S1bound}.

Our first result is the following.

\begin{theorem}					\label{thm:H}
The limit $H$ in~\eqref{eq:H} exists.
Also,
\begin{equation*}			
H = \int_{2\pi}^\infty\frac{Q(t)}{t^2}\dif t - \frac{1}{16\pi}\,,
\end{equation*}
where $Q(T) = N(T)-L(T)$ is as above.
\end{theorem}
\begin{proof}
Suppose that $2\pi \le T\not\in\NZ$.
Using Stieltjes integrals, and noting that
$\gamma_1 > 2\pi$ and $Q(2\pi) = \frac18$, we have
\begin{align}
\nonumber
G(T) &= \sum_{0 < \gamma \le T}\frac{1}{\gamma} 
 = \int_{2\pi}^T \frac{\dif N(t)}{t}
 = \int_{2\pi}^T \frac{\dif L(t)}{t} + \int_{2\pi}^T \frac{\dif Q(t)}{t}\\
\nonumber
&= \frac{1}{2\pi}\int_{2\pi}^T\frac{\log(t/2\pi)}{t}\dif t
 + \left[\frac{Q(t)}{t} + \int\frac{Q(t)}{t^2}\dif t\right]_{2\pi}^T\\
							\label{eq:lem1}
&= \frac{\log^2(T/2\pi)}{4\pi} + \frac{Q(T)}{T}-\frac{1}{16\pi}
 + \int_{2\pi}^T\frac{Q(t)}{t^2}\dif t\,.
\end{align}
Thus
\[
G(T) - \frac{\log^2(T/2\pi)}{4\pi} = 
\int_{2\pi}^T\frac{Q(t)}{t^2}\dif t -\frac{1}{16\pi} + O((\log T)/T).
\]
Letting $T \to \infty$, the last integral converges, so the limit of the
left-hand-side exists, and we obtain
\[
H = \lim_{T\to\infty} \left(G(T) - \frac{\log^2(T/2\pi)}{4\pi}\right)
  = \int_{2\pi}^\infty\frac{Q(t)}{t^2}\dif t - \frac{1}{16\pi}\,.
\]
This completes the proof.
\end{proof}

\section{Two lemmas}					\label{sec:lemmas}

We now give two lemmas 
that are used in the proof of Theorem~\ref{thm:accelerated}.

\begin{lemma}					\label{lem:Q-integral}
If $2\pi\le T\not\in\NZ$, 
then
\[
\int_{2\pi}^T\frac{Q(t)}{t^2}\dif t
= G(T) - \frac{Q(T)}{T} + \frac{1}{16\pi} - \frac{\log^2(T/2\pi)}{4\pi}\,.
\]
\end{lemma}
\begin{proof}
This is just a rearrangement of~\eqref{eq:lem1}
in the proof of Theorem~\ref{thm:H}.
\end{proof}

\begin{lemma}				\label{lem:E2-bound}
If $T \ge 2\pi$ and 
\begin{equation}			\label{eq:E2}
E_2(T) := \int_{T}^\infty \frac{Q(t)}{t^2}\dif t,
\end{equation}
then
\[
|E_2(T)| \le \frac{4.27 + 0.12\log T}{T^2}\,.
\]
\end{lemma}
\begin{proof}
To bound $E_2(T)$ we note that, from~\eqref{eq:Q},
\begin{equation}				\label{eq:alpha}
\int_T^\infty\frac{Q(t)}{t^2}\,dt
= \int_T^\infty\frac{S(t)}{t^2}\,dt
  + \frac{0.1\vartheta}{T^2}\,.
\end{equation}
Also, using integration by parts, 
\begin{equation}				\label{eq:parts}
\int_T^\infty \frac{S(t)}{t^2}\,dt
 = -\frac{S_1(T)-c}{T^2} + 2\int_T^\infty\frac{S_1(t)-c}{t^3}\,dt\,.
\end{equation}
Using~\eqref{eq:S1bound}, we have
\begin{align}
\nonumber
\left|\int_T^\infty \frac{S(t)}{t^2}\,dt\right|
&\le \frac{|S_1(T)-c|}{T^2} + 2\int_T^\infty \frac{|S_1(t)-c|}{t^3}\,dt\\
\nonumber
&\le \frac{A_0+A_1\log T}{T^2}
 + 2\int_T^\infty \frac{A_0+A_1\log t}{t^3}\,dt\\
						\label{eq:beta}
&= \frac{2A_0+0.5A_1+2A_1\log T}{T^2}\,.
\end{align}
Using \eqref{eq:alpha}, this gives
\[
|E_2(T)| \le \frac{2A_0+0.5A_1+0.1+2A_1\log T}{T^2}\,. 
\]
Inserting the values $A_0=2.067$ and $A_1=0.059$ gives
the result.
\end{proof}

We note that the bound \eqref{eq:beta} might
be improved by using a result of Fujii~\cite[Thm.~2]{Fujii} to bound 
the integral
of $S_1(t)/t^3$ in~\eqref{eq:parts}, although we are not aware of any
explicit version of Fujii's estimate.
The bound would then be dominated by the term $-S_1(T)/T^2$
in \eqref{eq:parts}.  This term is $o((\log T)/T^2)$ iff the
Lindel\"of Hypothesis (LH) is true,
see~\cite[Thm~13.6(B) and Note 13.8]{Titch}.
Thus, obtaining an order-of-magnitude improvement in the bound
on $E_2(T)$ is equivalent to proving LH. 
\section{Numerical approximation of $H$}	\label{sec:extrapolation}

We consider two methods to approximate $H$ numerically. The first
method truncates the sum and integral in the definition~\eqref{eq:H}
at height $T\ge 2\pi e$, giving an approximation with error
$E(T) = O((\log T)/T)$. An explicit bound
\begin{equation}				\label{eq:simple-approx}
H = G(T) - \frac{\log^2(T/2\pi)}{4\pi} + 
 A\vartheta\left(\frac{2\log T+1}{T}\right)
\end{equation}
follows from Lehman~\cite[Lem.~1]{Lehman}.
Lehman gave $A=2$, but from~\cite[Cor.~1]{BPTCv7} we may take $A = 0.28$.
Thus, we can obtain about $5$ decimal places by summing over
the first $10^6$ zeros of $\zeta(s)$, i.e. to height
$T = 600270$.  In this manner we find $H \approx -0.01716$.
It is difficult to obtain many more correct digits
because of the slow convergence.
However, the result is sufficient to show
that $H$ is negative, 
which is significant in the proof of~\cite[Lem.~8]{BPTCv7}.

Convergence can be accelerated using Theorem~\ref{thm:accelerated},
which improves the error bound $E(T) = O((\log T)/T)$ 
of \eqref{eq:simple-approx}
to $E_2(T) = O((\log T)/T^2)$. 
Note that the error term $E_2(T)$
is a continuous function of~$T$. 
This is unlike $E(T)$, which has jumps for $T\in\NZ$.

\begin{theorem}				\label{thm:accelerated}
For all $T \ge 2\pi$,
\begin{equation}			\label{eq:Hsmooth}
H = \sum_{0<\gamma\le T}\left(\frac{1}{\gamma}-\frac{1}{T}\right)
    -\frac{\log^2(T/2\pi e)+1}{4\pi} 
    + \frac{7}{8T} + E_2(T)\,,
\end{equation}
where $E_2(T)$ is as in~\eqref{eq:E2}, and
$|E_2(T)| \le (4.27 + 0.12\log T)/T^2$.
\end{theorem}
\begin{proof}
First assume that $T\not\in\NZ$.
{From} Theorem~\ref{thm:H} and Lemma~\ref{lem:Q-integral},
\[
H = G(T) - \frac{Q(T)}{T} - \frac{\log^2(T/2\pi)}{4\pi} + E_2(T),
\]
but $Q(T) = N(T) - L(T)$, so
\[
H = \sum_{0<\gamma\le T}\left(\frac{1}{\gamma}-\frac{1}{T}\right)
  + \frac{\log(T/2\pi)-1}{2\pi} + \frac{7}{8T}
  - \frac{\log^2(T/2\pi)}{4\pi} + E_2(T).
\]
Simplification gives~\eqref{eq:Hsmooth},
and a continuity argument shows that~\eqref{eq:Hsmooth}
holds if $T\in\NZ$.
Finally, the bound on $E_2(T)$ follows from Lemma~\ref{lem:E2-bound}.
\end{proof}

\begin{corollary}\label{cor:huck}
Let $H$ be defined by \eqref{eq:H}. We have
\begin{equation*}
H = 
-0.0171594043070981495 + \vartheta(10^{-18}).		\label{eq:Happrox}
\end{equation*}
\end{corollary}
\begin{proof}
This follows from Theorem~\ref{thm:accelerated}
by an interval-arithmetic computation using the first $n = 10^{10}$ zeros, 
with $T = \gamma_n \approx 3293531632.4\,$. 
\end{proof}

To illustrate Theorem~\ref{thm:accelerated}, we give some
numerical results in Table~\ref{tab:Hest}. 
The first column ($n$) gives the number of zeros used, and
the second column is the estimate of $H$ obtained from~\eqref{eq:Hsmooth},
using $T=\gamma_n$.
The first incorrect digit of each entry is underlined.
\begin{table}[h]
\begin{center}
\begin{tabular}{ | c | c | } \hline
  $n$ & $H$ estimate \\ \hline
      10 & $-0.017\underbar{3}72393877$\\
     100 & $-0.017159\underbar{7}65533$\\
    1000 & $-0.017159\underbar{6}03500$\\
   10000 & $-0.017159404\underbar{8}75$\\
  100000 & $-0.017159404\underbar{2}44$\\
 1000000 & $-0.017159404307$\\
\hline
\end{tabular}
\end{center}
\vspace*{-10pt}
\caption{Numerical estimation of $H$ using Theorem~\ref{thm:accelerated}.}
\vspace*{0pt}
\label{tab:Hest}
\end{table}

\section{Related results in the literature}		\label{sec:related}

B\"uthe~\cite[Lem.~3]{Buthe-2016} gives the inequality
\begin{equation}					\label{eq:Gineq}
G(T) \le \frac{\log^2(T/2\pi)}{4\pi}
 \text{ for } T \ge 5000.
\end{equation}
In~\cite[Lem.~8]{BPTCv7}, we give a different proof of 
\eqref{eq:Gineq},
and show that it holds for $T \ge 4\pi e$.

Hassani~\cite{Hassani-2016} shows (in our notation) that
\[
G(T) = \frac{\log^2(T/2\pi)}{4\pi} + O(1),
\]
and gives numerical bounds for the ``$O(1)$'' term.
A similar bound is given in~\cite[Lem.~2.10]{Saouter}.
Hassani does not prove existence of the limit~\eqref{eq:H},
but asks (see \cite[p.~114]{Hassani-2016}) whether it exists.
We have answered this in our Theorem~\ref{thm:H}.

In fact, Hassani works with
\[
\Delta_N := \sum_{n=1}^N \frac{1}{\gamma_n}
 - \left(\frac{1}{4\pi}\log^2\gamma_N -
   \frac{\log(2\pi)}{2\pi}\log\gamma_N\right),
\]
so in our notation
\[
\Delta_N = G(\gamma_N) -\frac{\log^2(\gamma_N/2\pi)}{4\pi}
 + \frac{\log^2(2\pi)}{4\pi}\,.
\]
Thus, the (hypothetical) limit to which Hassani refers is,
in our notation,
\[
H + \frac{\log^2(2\pi)}{4\pi} = 0.2516367513127059665 + \vartheta(10^{-18}).
\]
This is consistent with the value $0.25163$ that Hassani gives
based on his calculations using $2\cdot 10^6$ nontrivial
zeros. Hassani also uses an averaging technique to obtain values
in the range $[0.2516372,0.2516375]$, but apparently decreasing, without
an obvious limit. The acceleration technique of
Theorem~\ref{thm:accelerated} is more effective, and has the virtue of
giving a rigorous error bound.

\subsection*{Acknowledgements}
TST is supported by ARC Grants DP160100932 and FT160100094;
DJP is supported by ARC Grant DP160100932 and EPSRC Grant EP/K034383/1.

\pagebreak[3]

\end{document}